\documentclass[12pt]{article}
\usepackage{mathtools}
\usepackage{amsthm}
\usepackage{csquotes}
\usepackage{amssymb}
\usepackage{graphicx}
\usepackage{subfig}
\usepackage{soul}  
\usepackage{xcolor}
\usepackage{empheq}
\usepackage{lipsum}
\usepackage{fullpage}
\usepackage{float}

\usepackage[top= 2cm, bottom = 2 cm, left = 2.2 cm, right= 2.2 cm]{geometry}  
\usepackage[obeyspaces,hyphens,spaces]{url}
\usepackage[unicode]{hyperref}
\hypersetup{
    colorlinks=true,
    linkcolor=blue,
    citecolor=blue,
    filecolor=magenta,
    urlcolor=cyan
}
\usepackage{listings}

\definecolor{codegreen}{rgb}{0,0.6,0}
\definecolor{codegray}{rgb}{0.5,0.5,0.5}
\definecolor{codepurple}{rgb}{0.58,0,0.82}
\definecolor{backcolour}{rgb}{0.95,0.95,0.92}

\lstdefinestyle{mystyle}{
	backgroundcolor=\color{backcolour},   
	commentstyle=\color{codegreen},
	keywordstyle=\color{magenta},
	numberstyle=\footnotesize\color{codegray},
	stringstyle=\color{codepurple},
	basicstyle=\ttfamily\small,
	breakatwhitespace=false,         
	breaklines=true,                 
	captionpos=b,                    
	keepspaces=true,                 
	numbers=left,                    
	numbersep=5pt,                  
	showspaces=false,                
	showstringspaces=false,
	showtabs=false,                  
	tabsize=2
}

\definecolor{seagreen}{rgb}{0.18, 0.55, 0.34}
\definecolor{mediumviolet-red}{rgb}{0.78, 0.08, 0.52}
\definecolor{khaki}{rgb}{0.94, 0.9, 0.55}

\lstdefinelanguage{mypython}
{
	keywords=[1]{from, import, assert, not, print},
	keywordstyle=[1]{\color{mediumviolet-red}},
	keywords=[2]{surecr, torch, cp, lo, pl},
	keywordstyle=[2]{\color{seagreen}},
	numbers=none,
	upquote=true,
	showstringspaces=false,
	basicstyle=\ttfamily,
	columns=fullflexible,
	keepspaces=true,
	emph={True,False,as,def,return,float,class,match,switch,len},
	emphstyle={\color{seagreen}},
	frame=trBL,
	belowskip=1em,
	aboveskip=1em,
	captionpos=b
}

\usepackage[capitalize, nameinlink, noabbrev]{cleveref}
\crefname{equation}{}{}
\crefname{chapter}{Chapter}{Chapters}
\crefname{item}{item}{items}
\crefname{figure}{Figure}{Figures}
\crefname{theorem}{Theorem}{Theorems}
\crefname{assumption}{Assumption}{Assumption}
\crefname{lemma}{Lemma}{Lemmas}
\crefname{proposition}{Proposition}{Propositions}
\crefname{corollary}{Corollary}{Corollarys}
\crefname{definition}{Definition}{Definitions}
\crefname{fact}{Fact}{Facts}
\crefname{example}{Example}{Examples}
\crefname{algorithm}{Algorithm}{Algorithms}
\crefname{remark}{Remark}{Remarks}
\crefname{note}{Note}{Notes}
\crefname{notation}{Notation}{Notations}
\crefname{case}{Case}{Cases}
\crefname{exercise}{Exercise}{Exercises}
\crefname{question}{Question}{Questions}
\crefname{claim}{Claim}{Claims}
\crefname{enumi}{}{}

 \usepackage{textcomp}

\numberwithin{equation}{section}

\usepackage{amsmath,xparse}
\makeatletter
\NewDocumentCommand{\lplabel}{o m}{%
	\makebox[0pt][r]{#2\hspace*{2em}}%
	\IfNoValueF{#1}
	{\def\@currentlabel{#2}\ltx@label{#1}}
}
\makeatother

\theoremstyle{plain}
\newtheorem{theorem}{Theorem}[section]

\newtheorem{fact}[theorem]{Fact}
\newtheorem{lemma}[theorem]{Lemma}
\newtheorem{proposition}[theorem]{Proposition}

\newtheorem{assumption}{Assumption}

\theoremstyle{definition}

\newtheorem{example}[theorem]{Example}

\usepackage{enumitem}

\newcommand{\minimize}{\ensuremath{\operatorname{minimize}}}
\newcommand{\maximize}{\ensuremath{\operatorname{maximize}}}
\newcommand{\argmin}{\ensuremath{\operatorname{argmin}}}

\newcommand{\Pro}{\ensuremath{\operatorname{P}}}

\newcommand{\Prox}{\ensuremath{\operatorname{Prox}}}

\newcommand{\dom}{\ensuremath{\operatorname{dom}}}

\providecommand{\abs}[1]{\left|#1\right|}
\providecommand{\norm}[1]{\left\lVert#1\right\rVert}
\providecommand{\innp}[1]{\left\langle#1\right\rangle}

\newcommand\scalemath[2]{\scalebox{#1}{\mbox{\ensuremath{\displaystyle #2}}}}

\begin{document}

\title{A Note on the Convergence of the OGAProx}

\author{
	Hui Ouyang\thanks{Department of Electrical Engineering, Stanford University.
		E-mail: \href{mailto:houyang@stanford.edu}{\texttt{houyang@stanford.edu}}.}
}

\date{October 30, 2023}

\maketitle

\begin{abstract}
In this note, we consider  the Optimistic Gradient Ascent-Proximal Point Algorithm 
(OGAProx)
 proposed   by Bo{\c{t}}, Csetnek, and Sedlmayer for solving a saddle-point problem 
 associated with a convex-concave function 
constructed by a nonsmooth coupling function and one regularizing function.
 
 We first provide a counterexample to show that the convergence of the minimax gap
 function, evaluated at the ergodic sequences, is insufficient to demonstrate 
 the convergence of the function values evaluated at the ergodic 
 sequences. 
 Then under the same assumptions used by Bo{\c{t}} et al.\,for proving the
 convergence of the minimax gap function, 
  we present convergence results 
  for the function values evaluated at the ergodic sequences generated by the OGAProx
  with convergence rates of order 
  $\mathcal{O}\left(\frac{1}{k}\right)$, $\mathcal{O}\left(\frac{1}{k^{2}}\right)$,
  and $\mathcal{O}\left(\theta^{k}\right)$ with $\theta \in (0,1)$ 
  for the associated convex-concave coupling function being convex-concave, 
  convex-strongly concave, and strongly convex-strongly concave, respectively. 
\end{abstract}

{\small
	\noindent
	{\bfseries 2020 Mathematics Subject Classification:}
	{
		Primary 90C25, 47H05;  
		Secondary 47J25,  90C30.
	}
	
	\noindent{\bfseries Keywords:}
	Convex-Concave Saddle-Point Problems, 
	Proximity Mapping, Gradient Ascent, 
	Convergence,   Linear Convergence
}

\section{Introduction} 
Throughout this work, let $\mathcal{H}_{1}$ and $\mathcal{H}_{2}$ be real Hilbert 
spaces, and let
 $f: \mathcal{H}_{1} \times \mathcal{H}_{2} \to \mathbf{R} \cup \{ -\infty, +\infty\}$ 
 satisfy that $(\forall y \in Y)$ $f(\cdot, y) : \mathcal{H}_{1}  \to \mathbf{R} \cup \{ 
 -\infty\}$ is proper, convex, and lower semicontinuous,  and that $(\forall x \in X)$
 $f(x, \cdot) : \mathcal{H}_{2} \to \mathbf{R} \cup \{ +\infty\}$ is proper, concave, and 
 upper semicontinuous.   
 We say  $(x^{*},y^{*}) \in  \mathcal{H}_{1} \times \mathcal{H}_{2}$  is a 
 \emph{saddle-point}  of $f$ if
 \begin{align} \label{eq:saddle-point}
\left(\forall (x,y) \in  \mathcal{H}_{1} \times \mathcal{H}_{2}\right) \quad  f(x^{*},y) \leq  
f(x^{*},y^{*}) \leq  f(x,y^{*}).
\end{align}
In this work, we assume that there exists at least one saddle-point of $f$, 
 and we aim to solve the following  \emph{convex-concave saddle-point problem}:
 \begin{align} \label{eq:problem}
 \maximize_{y \in \mathcal{H}_{2}}	\minimize_{x \in \mathcal{H}_{1}} f(x,y).
 \end{align}
 
 In the paper \cite{BotCsetnekSedlmayer2022accelerated} by Bo{\c{t}},
 Csetnek, and Sedlmayer, 
 the authors proposed the Optimistic Gradient Ascent-Proximal Point Algorithm 
 (OGAProx)
 for solving a saddle-point problem 
 associated with a convex-concave function 
 with a nonsmooth coupling function and one regularizing function.
 In particular, the authors proved convergence results on 
 the sequence of the iterations  
 and also the minimax gap function evaluated at the ergodic sequences 
 for the OGAProx.
 
In this work, 
under the same assumptions used by Bo{\c{t}} et al.\,for showing the
convergence of the minimax gap function, 
we shall complement convergence results on the values of function 
evaluated at the ergodic sequences generated by  
the OGAProx.
 
The rest of the work is organized as follows. 
In \cref{section:counterexample}, we work on one example 
showing that the convergence of 
the minimax gap function evaluated at the ergodic sequences
doesn't imply the convergence of 
the values of function evaluated at the ergodic sequences for the OGAProx. 
In \cref{section:convergence}, 
under the same assumptions used by Bo{\c{t}} et al.\,for showing the
convergence of the minimax gap function, 
we provide the convergence 
of the values of function evaluated at the ergodic sequences generated by the 
OGAProx. 
 In particular, as the authors did in \cite{BotCsetnekSedlmayer2022accelerated}, 
 we consider
 three cases of the associated convex-concave function
 (convex-concave, convex-strongly concave, and strongly convex-strongly concave),
 and show the convergence of 
 the values of function evaluated at the ergodic sequences generated by the OGAProx
 with convergence rates of
 order $\mathcal{O}\left(\frac{1}{k}\right)$, $\mathcal{O}\left(\frac{1}{k^{2}}\right)$, and
 $\mathcal{O}\left(\theta^{k}\right)$ with $\theta \in (0,1)$, respectively.

To end this section, we provide some notation frequently used in this work below. 
 We use the convention that $\mathbf{N}:= \{0,1, 2, \cdots \}$ is the set of all 
 nonnegative integers.
 $\mathbf{R}$,  $\mathbf{R}_{+}$, and  $\mathbf{R}_{++}$ are the set
 of all real numbers, the set of all nonnegative real numbers, 
 and the set of all positive real numbers, respectively.  
 Let  $\mathcal{H}$ be a real Hilbert space. 
 Let  $g: \mathcal{H} \to \mathbf{R} \cup \{+\infty\}$ be a proper, convex, 
 and lower semicontinuous function. 
 The \emph{proximity operator $\Prox_{g}$ of $g$} is defined by
 \begin{align*}
 	\Prox_{g}: \mathcal{H} \to \mathcal{H}: x \mapsto \argmin_{y \in \mathcal{H}} 
 	\left(g(y) + \frac{1}{2} \norm{x-y}^{2}\right).
 \end{align*}
 
 \section{Counterexample} \label{section:counterexample}
 In this section, we consider the function $f: \mathbf{R}^{2} \to \mathbf{R}$ defined as 
 \begin{align*}
 	(\forall (x,y) \in \mathbf{R}^{2}) \quad f(x,y) = xy.
 \end{align*}

It is easy to see that the saddle-point of $f$ is $(x^{*},y^{*}) =(0,0)$. 
Moreover, we have that 
\begin{align} \label{eq:minmaxgap}
	(\forall (x,y) \in \mathbf{R}^{2}) \quad f(x,y^{*}) -f(x^{*},y) =0 -0 =0 
	=f(0,0)=f(x^{*},y^{*}).
\end{align}

Note that $(\forall (\bar{x},\bar{y}) \in \mathbf{R}^{2})$ $\nabla_{x}f(\bar{x},\bar{y})  = 
\bar{y}$ and $\nabla_{y}f(\bar{x},\bar{y})  = \bar{x}$. Then to satisfy 
\cite[Inequality~(2)]{BotCsetnekSedlmayer2022accelerated} (that is also 
\cref{eq:nablay} below), we can take $L_{yx} = 1$ and $L_{yy}=0$. 
Clearly, this function with $(\forall (x,y) \in \mathbf{R}^{2})$  $\Phi(x,y) =xy$ and $g(y) 
 \equiv 0$ satisfies all assumptions of the convex-concave function presented in 
 \cite[Section~1.1]{BotCsetnekSedlmayer2022accelerated} 
 (which is provided in \cref{assumption} below).

Let $(x^{0},y^{0}) $ be in $\mathbf{R}^{2}$. 
Based on \cite[Section~1.2]{BotCsetnekSedlmayer2022accelerated}, 
 the sequence of iterations $\left((x^{k},y^{k})\right)_{k \in \mathbf{N}}$ generated by 
 the OGAProx (see also \cref{eq:OGAProx} 
 below for details) is: for every $k \in \mathbf{N}$,
 
 \begin{subequations}\label{eq:yk}
 	\begin{align}
 		y^{k+1}  &= \Prox_{\sigma_{k}0} \left(y^{k} + \sigma_{k} \left( (1+\theta_{k})x^{k} 
 		-\theta_{k}x^{k-1}   \right) \right) \\
 		&=\argmin_{y \in \mathbf{R}} \left(  0 
 		+\frac{1}{2\sigma_{k}} \abs{y^{k} + \sigma_{k} \left( (1+\theta_{k})x^{k} 
 			-\theta_{k}x^{k-1}   \right) -y   }^{2} \right)\\
 		&= y^{k} + \sigma_{k} \left( (1+\theta_{k})x^{k} 
 		-\theta_{k}x^{k-1}   \right) 
 	\end{align}
 \end{subequations}
 and 
 	\begin{align}\label{eq:xk}
 		x^{k+1} = \Prox_{\tau_{k}f(\cdot, y^{k+1})} x^{k}
 		= \argmin_{x \in \mathbf{R}} xy^{k+1} +\frac{1}{2\tau_{k}} \abs{x-x^{k}}^{2}
 		=x^{k}-\tau_{k}y^{k+1}.
 	\end{align}

 	By continuing applying formulae above, we know that 
 	\begin{align}\label{eq:yxk}
 		(\forall k \in \mathbf{N}) \quad 
 		y^{k+1} =y^{0} +\sum^{k}_{i=0}\sigma_{i} \left( (1+\theta_{i})x^{i} 
 		-\theta_{i}x^{i-1}   \right) \quad \text{and} \quad x^{k+1} =x^{0} - 
 		\sum^{k}_{i=0}\tau_{i}y^{i+1}.
 	\end{align}
Set
 	\begin{align}\label{eq:fhatxyk}
 		(\forall k \in \mathbf{N} \smallsetminus \{0\}) \quad 
 		\hat{x}_{k} := \frac{1}{\sum^{k-1}_{i=0}t_{i}} 
 		\sum^{k-1}_{j=0}t_{j}x^{j+1} 
 		\quad \text{and} \quad 
 		\hat{y}_{k} := \frac{1}{\sum^{k-1}_{i=0}t_{i}} 
 		\sum^{k-1}_{j=0}t_{j}y^{j+1},
 	\end{align}
where $(\forall k \in \mathbf{N})$ $t_{k} \in \mathbf{R}_{++}$.

 We show below in \cref{example:counterexample} that the convergence 
 $ \lim_{k \to  \infty}	f( \hat{x}_{k}, y^{*} ) -f(x^{*},  \hat{y}_{k} )  =  f(x^{*},y^{*})$ 
 does not  imply the convergence
 of $\lim_{k \to \infty} f(\hat{x}_{k},  \hat{y}_{k} )  = f(x^{*},y^{*})$.

 \begin{example} \label{example:counterexample}
 Let $f: \mathbf{R}^{2} \to \mathbf{R}$ defined as 
 \begin{align*}
 	(\forall (x,y) \in \mathbf{R}^{2}) \quad f(x,y) = xy.
 \end{align*}
Let $\left( (x^{k},y^{k}) \right)_{k \in \mathbf{N}}$ be defined as \cref{eq:yk} and 
\cref{eq:xk} with $(x^{-1},y^{-1}) =(x^{0},y^{0}) =(1,1)$.
  Let $(\forall k \in \mathbf{N})$ $t_{k} \equiv t_{0} \in \mathbf{R}_{++}$. 
  Then, via \cref{eq:fhatxyk},  we have that 
  	\begin{align*}
  	(\forall k \in \mathbf{N} \smallsetminus \{0\}) \quad 
  	\hat{x}_{k} = \frac{1}{k} 
  	\sum^{k-1}_{j=0}x^{j+1} 
  	\quad \text{and} \quad 
  	\hat{y}_{k} = \frac{1}{k} 
  	\sum^{k-1}_{j=0}y^{j+1}.
  \end{align*}
Moreover, the following statements hold.
 \begin{enumerate}
 	\item \label{example:counterexample:minimaxgap} $(\forall k \in \mathbf{N} 
 	\smallsetminus \{0\})$ $f(\hat{x}_{k},y^{*} )- 
 	f(x^{*},\hat{y}_{k} ) \equiv 0$.
 	
 	Consequently, 
 	$ \lim_{k \to  \infty}	f( \hat{x}_{k}, y^{*} ) -f(x^{*},  \hat{y}_{k} )  =  f(x^{*},y^{*})$. 
 	
 	\item \label{example:counterexample:assump1} Suppose $(\forall k \in \mathbf{N})$ 
 	$t_{k} \equiv t_{0}= 1$, $\tau_{k} \equiv \tau \in 
 	\mathbf{R}_{++}$,
 	$\sigma_{k} \equiv \sigma \in \mathbf{R}_{++}$, and 
 	$\theta_{k} \equiv 1$. Then $f( \hat{x}_{k} ,\hat{y}_{k} ) = \frac{1}{k^{2}} 
 	\frac{1}{\tau \sigma} \left( 
 	\left(x^{0}-x^{k} -\tau y^{k+1}\right)\left( y^{k+1} 
 	-y^{0} -\sigma x^{k}    \right) \right)\to f(x^{*},y^{*})$ as $k \to \infty$.
 	
 	\item \label{example:counterexample:assump2} 
 	Let   $\epsilon \in (0, \frac{6}{2\pi^{2}})$. 
 	Suppose that    
 	\begin{align}\label{eq:parameters}
 		(\forall k \in \mathbf{N}) \quad 
 		\sigma_{k} =\epsilon, \theta_{k} =\epsilon, \text{ and } \tau_{k} = \begin{cases}
 			\frac{\epsilon }{y^{k+1}(k+1)^{2}} &\quad \text{if } y^{k+1} \neq 0,\\
 			0 &\quad  \text{otherwise}.
 		\end{cases}
 	\end{align}	
  Then we have the following assertions. 
 		\begin{enumerate}
 			\item \label{example:counterexample:assump2:xyk}  $(\forall k \in \mathbf{N})$ 
 			$x^{k} >\frac{1}{2}$ and $y^{k} 
 			>1-\epsilon^{2}>0$.
 			\item \label{example:counterexample:assump2:f}  $	\left(k \in \mathbf{N} 
 			\smallsetminus \{0\}\right)	\quad f( \hat{x}_{k} 
 			,\hat{y}_{k} ) 
 			=\hat{x}_{k} \hat{y}_{k} > \frac{1-\epsilon^{2}}{2} >0$.
 			
 			Consequently,  $f( \hat{x}_{k} ,\hat{y}_{k} ) \not\to 
 			f(x^{*},y^{*})$ as $k \to \infty$. 
 		\end{enumerate}
 \end{enumerate}
 \end{example}

\begin{proof}
	\cref{example:counterexample:minimaxgap}: This is clear from \cref{eq:minmaxgap}. 

\cref{example:counterexample:assump1}:  Note that this assumption is consistent with 
that of \cite[Theorem~9]{BotCsetnekSedlmayer2022accelerated}.
Hence, due to results from \cite[Theorem~9]{BotCsetnekSedlmayer2022accelerated},
we know that under this assumption,
 the sequence of iterations $\left( (x^{k},y^{k}) \right)_{k \in \mathbf{N}}$ is bounded. 

Combine \cref{eq:yxk} with our assumptions of the parameters to deduce  that
 for every $k \in \mathbf{N}$,
	\begin{align*}
		&y^{k+1} =y^{0} +\sigma \sum^{k}_{i=0}x^{i} + \sigma \theta \sum^{k}_{i=0} \left(  
		x^{i} -x^{i-1}   \right) =y^{0} +\sigma \sum^{k}_{i=0}x^{i} + \sigma \theta 
		\left( x^{k} -x^{0}   \right), \text{ and}   \\
		& x^{k+1} =x^{0} - \tau \sum^{k}_{i=0} y^{i+1},
	\end{align*}
which yield that 
\begin{subequations} \label{eq:symxyk}
\begin{align}
&x^{0} + \sum^{k-1}_{j=0}x^{j+1} =	\sum^{k}_{i=0}x^{i}  = \frac{1}{\sigma} \left(y^{k+1} 
-y^{0}\right) +\theta \left(x^{0} 
	-x^{k}\right), \text{ and}   \label{eq:symxyk:x}\\
&\sum^{k-1}_{i=0} y^{i+1}  +y^{k+1} = \sum^{k}_{i=0} y^{i+1} =\frac{1}{\tau} \left(x^{0} 
	-x^{k+1}\right). \label{eq:symxyk:y}
\end{align}	
\end{subequations}

Therefore, we obtain that  for every $k \in \mathbf{N}$,
\begin{align*}
	 &\hat{x}_{k} =\frac{1}{k}  \sum^{k-1}_{j=0}x^{j+1} 
	 = \frac{1}{k} \left(  \frac{1}{\sigma} \left(y^{k+1} 
	 -y^{0}\right) +\left(x^{0} 
	 -x^{k}\right) -x^{0}  \right) 
	 =\frac{1}{\sigma k}\left(y^{k+1}  -y^{0}\right)  -\frac{1}{k}x^{k};\\
	 &\hat{y}_{k} =\frac{1}{k}  \sum^{k-1}_{j=0}y^{j+1} 
	 =   \frac{1}{k} \left(  \frac{1}{\tau} \left(x^{0}  -x^{k}\right) -y^{k+1} \right). 
\end{align*}
Therefore, 
\begin{align*}
	f( \hat{x}_{k} ,\hat{y}_{k} )= \hat{x}_{k} \hat{y}_{k} = \frac{1}{k^{2}} \frac{1}{\tau 
	\sigma} \left( 
	\left(x^{0}-x^{k} -\tau y^{k+1}\right)\left( y^{k+1} 
	-y^{0} -\sigma x^{k}    \right) \right),
\end{align*}
which, 
combined with the boundedness of $\left( (x^{k},y^{k}) \right)_{k \in \mathbf{N}}$,
guarantees that $\lim_{k \to \infty} f( \hat{x}_{k} ,\hat{y}_{k} ) = 0=f(x^{*},y^{*})$.

\cref{example:counterexample:assump2}: 
Note that via \cref{eq:yk}  for every $k \in \mathbf{N}$,
we are able to take $\sigma_{k}$ and $\theta_{k}$ 
based on the values of $(\forall i \in \{0, \cdots, k\})$ $x^{i}$.
Similarly, according to \cref{eq:xk}, for every $k \in \mathbf{N}$,
we are able to take $\tau_{k}$ based on the values of $(\forall i \in \{0, \cdots, k, 
k+1\})$ $y^{i}$. Therefore, our assumption is practical.

\cref{example:counterexample:assump2:xyk}:  We prove below by induction that  
$(\forall k \in \mathbf{N})$ $x^{k} >\frac{1}{2}$ and $y^{k} >1-\epsilon^{2}$.
Recall that $(x^{-1},y^{-1}) =(x^{0},y^{0}) =(1,1)$. Combine \cref{eq:yxk} and 
\cref{eq:parameters} to derive that 
\begin{align}
	y^{1} = y^{0} +\epsilon x^{0} +\epsilon^{2} \cdot 0= 1+\epsilon  > 1-\epsilon^{2}>0 
	\text{ 
	and }
	x^{1} = x^{0} -  \frac{\epsilon}{y^{1}} y^{1} =x^{0} -\epsilon =1-\epsilon >\frac{1}{2}.
\end{align}
Let $N \in \mathbf{N} \smallsetminus \{0\}$. 
Assume that   $\left( \forall i \in \{0, 1, \cdots, N \}\right)$  
$x^{i} >\frac{1}{2}$ and $y^{i} >1-\epsilon^{2} >0$.
Applying \cref{eq:yxk} and \cref{eq:parameters} again, we have that
\begin{align*} 
	y^{N+1} &=y^{0} +\epsilon\sum^{N}_{i=0}x^{i} + \epsilon^{2}\sum^{N}_{i=0} 
	\left(  x^{i}-x^{i-1} \right)\\ 
	&=y^{0} +\epsilon\sum^{N}_{i=0}x^{i} + \epsilon^{2} \left(  x^{N}-x^{0} \right) \\
	&= y^{0} -\epsilon^{2} x^{0} + \epsilon\sum^{N}_{i=0}x^{i} + \epsilon^{2} x^{N}\\ 
	&>y^{0} -\epsilon^{2} x^{0} =1-\epsilon^{2}>0.
\end{align*} 
Employing this result, \cref{eq:yxk}, and \cref{eq:parameters}, we derive that
\begin{align*} 
	x^{N+1} &=x^{0} - 
	\sum^{N}_{i=0}\tau_{i}y^{i+1} =x^{0} -\epsilon 
	\sum^{N}_{i=0}\frac{1}{y^{i+1}(i+1)^{2}}y^{i+1} \\
	&=x^{0}-\epsilon \sum^{N}_{i=0} 
	\frac{1}{(i+1)^{2}} > x^{0} - \epsilon \sum_{i \in \mathbf{N}} 
	\frac{1}{(i+1)^{2}} =1-\epsilon \frac{\pi^{2}}{6} >1-\frac{1}{2} = \frac{1}{2}.
\end{align*}
Altogether, we proved that 
$(\forall k \in \mathbf{N})$ $x^{k} >\frac{1}{2}$ and $y^{k} >1-\epsilon^{2} >0$
by induction.

\cref{example:counterexample:assump2:f}: 
According to \cref{example:counterexample:assump2:xyk}, we have that
for every $k \in \mathbf{N} \smallsetminus \{0\}$,
\begin{align*}
	\hat{x}_{k}  
	=\frac{1}{k} \sum^{k-1}_{j=0}x^{j+1}  >\frac{1}{2}
	\quad \text{and} \quad
	\hat{y}_{k}  
	=\frac{1}{k} \sum^{k-1}_{j=0}y^{j+1}>1-\epsilon^{2},
\end{align*}
which guarantees that 
\begin{align*}
	\left(k \in \mathbf{N} \smallsetminus \{0\}\right)	\quad f( \hat{x}_{k} ,\hat{y}_{k} ) 
	=\hat{x}_{k} \hat{y}_{k} > \frac{1-\epsilon^{2}}{2} >0.
\end{align*}
 This shows that it is impossible to have  $f( \hat{x}_{k} ,\hat{y}_{k} ) \to 
 f(x^{*},y^{*})$ as $k \to \infty$ in this case since $f(x^{*},y^{*}) =0$. 
\end{proof}
 
\section{Convergence of OGAProx} \label{section:convergence}

In this section, we shall work on the convergence of the values of function 
evaluated at the ergodic sequences constructed by the OGAProx.
In particular, we will consider three cases of the associated function: 
convex-concave, convex-strongly concave, and 
strongly convex-strongly concave.

The following result will play an essential role in proving our convergence results later.

\begin{fact}	\label{fact:fxkykx*y*} {\rm \cite[Lemma~2.5]{OuyangAPPA}}
	Let $f: \mathcal{H}_{1} \times \mathcal{H}_{2} \to \mathbf{R} \cup \{-\infty, +\infty\}$ 
	satisfy that $(\forall y \in \mathcal{H}_{2} )$ $f(\cdot, y)$ is convex 
	and $(\forall x \in \mathcal{H}_{1})$ $f(x,\cdot)$ is concave. 
	Let $(x^{*},y^{*})$ be a saddle-point of $f$, let $((x^{k},y^{k}))_{k \in \mathbf{N}}$ 
	be in $\mathcal{H}_{1}\times \mathcal{H}_{2}$, and let $(\forall k \in 
	\mathbf{N})$ $t_{k} \in \mathbf{R}_{+}$ with $t_{0} \in \mathbf{R}_{++}$.
	Set 
	\begin{align} \label{eq:lemma:fxkykx*y*}
	(\forall k \in \mathbf{N}) \quad 
	\hat{x}_{k} := \frac{1}{\sum^{k-1}_{i=0}t_{i}} 
	\sum^{k-1}_{j=0}t_{j}x^{j+1} 
	\quad \text{and} \quad 
	\hat{y}_{k} := \frac{1}{\sum^{k-1}_{i=0}t_{i}} 
	\sum^{k-1}_{j=0}t_{j}y^{j+1}.
	\end{align}
	Then we have that for every $k \in \mathbf{N}$,
	\begin{subequations}
		\begin{align}
			&f(\hat{x}_{k},\hat{y}_{k}) -f(x^{*},y^{*}) 
			\leq  \frac{1}{\sum^{k-1}_{i=0}t_{i}} 
			\sum^{k-1}_{j=0} t_{j} \left( f(x^{j+1}, \hat{y}_{k})  
			-f(x^{*},y^{j+1}) \right); \label{eq:lemma:fxkykx*y*:x}\\
			&f(x^{*},y^{*}) -f(\hat{x}_{k},\hat{y}_{k})
			\leq  \frac{1}{\sum^{k-1}_{i=0}t_{i}} 
			\sum^{k-1}_{j=0} t_{j} \left( f(x^{j+1},y^{*})  
			-f(\hat{x}_{k},y^{j+1}) \right).\label{eq:lemma:fxkykx*y*:y}
		\end{align}
	\end{subequations}
	Consequently, if
	\begin{align*}
		&\lim_{k \to \infty}   \frac{1}{\sum^{k-1}_{i=0}t_{i}} 
		\sum^{k-1}_{j=0}t_{j} \left( f(x^{j+1}, 
		\hat{y}_{k})  -f(x^{*},y^{j+1}) \right) =0, \text{ and}\\
		& \lim_{k \to \infty}  \frac{1}{\sum^{k-1}_{i=0}t_{i}} 
		\sum^{k-1}_{j=0}t_{j} \left( 
		f(x^{j+1},y^{*})  
		-f(\hat{x}_{k},y^{j+1}) \right)=0,
	\end{align*} 
	then 
	$\lim_{k \to \infty} f(\hat{x}_{k},\hat{y}_{k}) = f(x^{*},y^{*})$.
\end{fact}

In the rest of this work, let $\mathcal{H}_{1}$ and $\mathcal{H}_{2}$ be real Hilbert 
spaces,
let $\Phi: \mathcal{H}_{1} \times \mathcal{H}_{2} \to \mathbf{R} \cup \{+\infty\}$ be 
a coupling function with 
$\dom \Phi:= \{ (x,y) \in \mathcal{H}_{1} \times \mathcal{H}_{2} ~:~ \Phi(x,y) < +\infty  \} 
\neq \varnothing$,
let $g: \mathcal{H}_{2} \to \mathbf{R} \cup \{+\infty\}$ be a regulariser, and let
$f: \mathcal{H}_{1} \times \mathcal{H}_{2} \to \mathbf{R} \cup \{-\infty, +\infty\}$ be 
defined as
\begin{align*}
	\left(\forall (x,y)  \in \mathcal{H}_{1} \times \mathcal{H}_{2} \right)
	\quad f(x,y) :=\Phi(x,y) -g(y).
\end{align*}
Let $(x^{*},y^{*})$ be a saddle-point of $f$.
Set $\Pro_{\mathcal{H}_{1}} (\dom \Phi) := \{ u \in \mathcal{H}_{1} ~:~ \exists y \in 
\mathcal{H}_{2} \text{ such that } (u,y) \in \dom \Phi\}$.	 We copy assumptions 
presented in
 \cite[Section~1.1]{BotCsetnekSedlmayer2022accelerated} below.
\begin{assumption} \label{assumption}
Henceforth, we have the following assumptions given in 
\cite[Section~1.1]{BotCsetnekSedlmayer2022accelerated}.
	\begin{enumerate}
		\item $g$ is proper, lower semicontinuous, and convex with modulus $\nu \in 
		\mathbf{R}_{+}$
		, i.e., $g-\frac{\nu}{2}\norm{\cdot}^{2}$ is convex;
		\item $(\forall y \in \dom g)$ $\Phi(\cdot, y) : \mathcal{H}_{1} \to \mathbf{R} \cup 
		\{+\infty\}$
		 is proper, convex, and lower semicontinuous;
		 \item $\Pro_{\mathcal{H}_{1}} (\dom \Phi)$ is closed, 
		 and  $\left( \forall x \in \Pro_{\mathcal{H}_{1}} (\dom \Phi) \right)$ we have that
		 $\dom \Phi (x, \cdot) =\mathcal{H}_{2}$ and 
		 $\Phi(x, \cdot) : \mathcal{H}_{2} \to \mathbf{R}$ is convex and Fr\'echet 
		 differentiable;
		 \item There exist $L_{yx} \in \mathbf{R}_{+}$ and $L_{yy} \in \mathbf{R}_{+}$ 
		 such that for all 
		 $(x,y)$  and $(x',y')$ in $\Pro_{\mathcal{H}_{1}} (\dom \Phi) \times \dom g$,
		 \begin{align} \label{eq:nablay}
		 	\norm{\nabla_{y} \Phi(x,y) - \nabla_{y} \Phi(x',y')} \leq L_{yx} \norm{x-x'} 
		 	+L_{yy}\norm{y-y'}.
		 \end{align}
	\end{enumerate}
\end{assumption}

We state the Optimistic Gradient Ascent-Proximal Point Algorithm 
(OGAProx) proposed in \cite[Section~1.2]{BotCsetnekSedlmayer2022accelerated}
below. Let $(x^{0},y^{0}) $ be in $\Pro_{\mathcal{H}_{1}} (\dom \Phi) \times \dom g$
and set $(x^{-1},y^{-1}) :=(x^{0},y^{0})$. For every $k \in \mathbf{N}$,
\begin{subequations} \label{eq:OGAProx}
	\begin{align}
		&y^{k+1} =\Prox_{\sigma_{k}g}\left(y^{k} + \sigma_{k} \left[(1+\theta_{k}) 
		\nabla_{y}\Phi(x^{k},y^{k}) - \theta_{k} \nabla_{y} \Phi(x^{k-1},y^{k-1}) 
		\right]\right),\label{eq:OGAProx:y}\\
		&x^{k+1} =\Prox_{\tau_{k}f(\cdot, y^{k+1})}(x^{k}),\label{eq:OGAProx:x}
	\end{align}
\end{subequations}
where $(\sigma_{k})_{k \in \mathbf{N}}$ and $(\tau_{k})_{k \in \mathbf{N}}$ are in 
$\mathbf{R}_{++}$, and $(\theta_{k})_{k \in \mathbf{N}}$  is in $(0,1]$. 

From now on, $\left( (x^{k},y^{k}) \right)_{k \in \mathbf{N}}$ is the sequence of 
iterations generated by the OGAProx presented in \cref{eq:OGAProx} above. 
Let $(\forall k \in \mathbf{N})$ $t_{k} := \frac{\theta_{0}}{\theta_{0}\theta_{1}\cdots 
\theta_{k}}$. Set
\begin{align} \label{eq:hatxy}
(\forall k \in \mathbf{N} \smallsetminus \{0\}) \quad 
\hat{x}_{k} := \frac{1}{\sum^{k-1}_{i=0}t_{i}} 
\sum^{k-1}_{j=0}t_{j}x^{j+1} 
\quad \text{and} \quad 
\hat{y}_{k} := \frac{1}{\sum^{k-1}_{i=0}t_{i}} 
\sum^{k-1}_{j=0}t_{j}y^{j+1}.
\end{align}
Moreover, we denote 
\begin{align}\label{eq:qk}
	(\forall k \in \mathbf{N}) \quad q_{k} :=\nabla_{y}\Phi(x^{k},y^{k}) 
	-\nabla_{y}\Phi(x^{k-1},y^{k-1}). 
\end{align}
\subsection{Convex-(Strongly) Concave Setting}
In this subsection, we consider the convergence of
the sequence of iterations generated by the OGAProx under the assumption that
the coupling function $\Phi$ is convex-concave and
that the function $g$ is convex with modulus $\nu \geq 0$.

Note that if $\nu=0$, then the function 
$\left(\forall (x,y) \in \mathcal{H}_{1} \times \mathcal{H}_{2}\right)$ $f(x,y)=\Phi(x,y) 
-g(y)$  is convex-concave, 
and that if $\nu> 0$, then the function 
$\left(\forall (x,y) \in \mathcal{H}_{1} \times \mathcal{H}_{2}\right)$ $f(x,y)=\Phi(x,y) 
-g(y)$  is convex-strongly concave.

Throughout this subsection, 
set for every $(x,y) \in \mathcal{H}_{1} \times \mathcal{H}_{2}$ and 
for every $k \in \mathbf{N}$,
\begin{align*}
	a_{k}(x,y) :=  &\frac{1}{2\tau_{k}} \norm{x -x^{k}}^{2} 
	+\frac{1}{2\sigma_{k}}
	\norm{y -y^{k}}^{2} + \theta_{k} \innp{q_{k}, y^{k}-y} 
	+\theta_{k} \frac{L_{yx}}{2 \alpha_{k}} \norm{x^{k} -x^{k-1}}^{2}\\
	&+\theta_{k} \frac{L_{yy}}{2} \norm{y^{k} -y^{k-1}}^{2},\\
	b_{k+1}(x,y) :=  	&\frac{1}{2\tau_{k}} \norm{x -x^{k+1}}^{2} +\frac{1}{2}
	\left( \frac{1}{\sigma_{k}} +\nu\right)
	\norm{y -y^{k+1}}^{2} +  \innp{q_{k+1}, y^{k+1}-y} \\
	&+  \frac{L_{yx}}{2 \alpha_{k+1}} \norm{x^{k+1} -x^{k}}^{2}
	+  \frac{L_{yy}}{2} \norm{y^{k+1} -y^{k}}^{2},\\
	c_{k}:= &\frac{1}{2} \left( \frac{1}{\tau_{k}} - \frac{L_{yx}}{\alpha_{k+1}}\right)
	\norm{x^{k+1} -x^{k}}^{2} + \frac{1}{2}
	\left(\frac{1}{\sigma_{k}} - L_{yy} -\theta_{k} \left(L_{yx}\alpha_{k} 
	+L_{yy}\right)\right)\norm{y^{k+1} -y^{k}}^{2}.
\end{align*}

We borrow some results proved in \cite{BotCsetnekSedlmayer2022accelerated}
in the following fact, which will be used in our proofs later.
\begin{fact}
	 \label{fact:preliminaries}
	Let $\nu \geq 0$, let $c_{\alpha} > L_{yx} \geq 0$, and 
let $\theta_{0} =1$, and let $\tau_{0}$ and $\sigma_{0}$ be in $\mathbf{R}_{++}$
	 such that 
	\begin{align*}
		\left(c_{\alpha} L_{yx} \tau_{0} +2 L_{yy} \right)\sigma_{0}  <1.
	\end{align*}
Define
\begin{align*}
	(\forall k \in \mathbf{N}) \quad \theta_{k+1} := \frac{1}{\sqrt{1+ \nu \sigma_{k}}}, \quad
	\tau_{k+1} := \frac{\tau_{k}}{\theta_{k+1}}, \quad
	\text{and} \quad \sigma_{k+1} := \theta_{k+1}\sigma_{k}.
\end{align*}
Set 
\begin{align*}
	(\forall k \in \mathbf{N}) \quad \alpha_{k} :=
	\begin{cases}
		 c_{\alpha} \tau_{0} \quad &\text{if } k =0,\\
		 c_{\alpha} \tau_{k-1} \quad &\text{if } k \geq 1,
	\end{cases}
\end{align*}
and 
\begin{align*}
	\delta := \min \left\{1- \frac{L_{yx}}{c_{\alpha}}, 1- \left( c_{\alpha} L_{yx} \tau_{0} + 
	2L_{yy} \right)\sigma_{0} \right\}.
\end{align*}
Then the following statements hold.
\begin{enumerate}
\item \label{fact:preliminaries:sigma} 
 {\rm \cite[Proposition~6]{BotCsetnekSedlmayer2022accelerated} }
 $(\forall k \in 	\mathbf{N})$  $\tau_{k+1} \geq 
\frac{\tau_{k}}{\theta_{k+1}}$
and 
$\sigma_{k+1} \geq \frac{\sigma_{k}}{\theta_{k+1} (1+\nu \sigma_{k})}$. Furthermore, 
\begin{align*}
	\frac{1-\delta}{\tau_{k}} \geq \frac{L_{yx}}{\alpha_{k+1}} 
	\quad \text{and} \quad 
	\frac{1-\delta}{\sigma_{k}} \geq L_{yx} \alpha_{k}\theta_{k} + L_{yy}(1+\theta_{k}).
\end{align*}
\item  \label{fact:preliminaries:tk}  
 {\rm \cite[Proposition~6]{BotCsetnekSedlmayer2022accelerated} }
 $(\forall k \in 	\mathbf{N})$ 
$t_{k} =\frac{\theta_{0}}{\theta_{0}\theta_{1}\cdots \theta_{k}} =
 \frac{\tau_{k}}{\tau_{0}}$.
 
 	\item \label{fact:preliminaries:qk} 
 {\rm \cite[Inequality~(17)]{BotCsetnekSedlmayer2022accelerated}} 
Let $(x,y)$ be in $\mathcal{H}_{1} \times \mathcal{H}_{2}$. 
Then for every $k \in \mathbf{N}$,
 \begin{align*}
\scalemath{0.9}{ 	\abs{ \innp{q_{k}, y^{k}-y} } \leq \frac{L_{yx}}{2}\left(\alpha_{k} 
 	\norm{y-y_{k}}^{2} 
 	+ \frac{1}{\alpha_{k}}\norm{x^{k}-x^{k-1}}^{2}\right) 
 	+\frac{L_{yy}}{2} \left(\norm{y-y^{k}}^{2} +\norm{y^{k} -y^{k-1}}^{2}\right)}.
\end{align*}
 \item \label{fact:preliminaries:f} 
 {\rm \cite[Inequality~(18)]{BotCsetnekSedlmayer2022accelerated}} 
Let $(x,y)$ be in $\mathcal{H}_{1} \times \mathcal{H}_{2}$. 
Then $(\forall k \in \mathbf{N})$ 
 $	f(x^{k+1},y) -f(x,y^{k+1}) \leq a_{k}(x,y) -b_{k+1}(x,y) -c_{k}$.
 
 \item  \label{fact:preliminaries:fhatxy} 
 {\rm \cite[Inequality~(24)]{BotCsetnekSedlmayer2022accelerated}} 
Let $(x,y)$ be in $\mathcal{H}_{1} \times \mathcal{H}_{2}$. 

Then $(\forall k \in \mathbf{N} \smallsetminus \{0\})$ 
 $\sum^{k-1}_{i=0} t_{i} \left(  f(\hat{x}_{k}, y) -f(x,\hat{y}_{k})  \right) \leq  
 \frac{t_{0}}{2\tau_{0}} \norm{x -x^{0}}^{2} +\frac{t_{0}}{2\sigma_{0}}
 \norm{y -y^{0}}^{2} 
 -\frac{t_{k}}{2\tau_{k}} \norm{x -x^{k}}^{2} -\frac{t_{k}}{2} \left( \frac{1}{\sigma_{k}} - 
 \theta_{k} \left(L_{yx} \alpha_{k} 
 +L_{yy}\right) \right)\norm{y -y^{k}}^{2}$.
 
\item  \label{fact:preliminaries:bc} 
 {\rm \cite[Inequality~(23) and the expression 
 above]{BotCsetnekSedlmayer2022accelerated} }
Let $(x,y)$ be in $\mathcal{H}_{1} \times \mathcal{H}_{2}$. 
Then $(\forall k \in \mathbf{N})$ $t_{k}b_{k+1}(x,y) \geq t_{k+1}a_{k+1}(x,y)$. 
Furthermore,
\begin{align*}
	(\forall k \in \mathbf{N}) \quad c_{k} \geq \sigma \left( \frac{1}{2\tau_{k}} 
	\norm{x_{k+1}-x_{k}}^{2} +\frac{1}{2\sigma_{k}} \norm{y^{k+1} -y^{k}}^{2} \right) 
	\geq 0.
\end{align*}
\end{enumerate}
\end{fact}

\begin{proposition}
	\label{prop:finequalities}
		Let $\nu \geq 0$, let $c_{\alpha} > L_{yx} \geq 0$, and 
	let $\theta_{0} =1$, and let $\tau_{0}$ and $\sigma_{0}$ be in $\mathbf{R}_{++}$
	such that 
	\begin{align*}
		\left(c_{\alpha} L_{yx} \tau_{0} +2 L_{yy} \right)\sigma_{0}  <1.
	\end{align*}
	Define
	\begin{align*}
		(\forall k \in \mathbf{N}) \quad \theta_{k+1} := \frac{1}{\sqrt{1+ \nu \sigma_{k}}}, 
		\quad
		\tau_{k+1} := \frac{\tau_{k}}{\theta_{k+1}}, \quad
		\text{and} \quad \sigma_{k+1} := \theta_{k+1}\sigma_{k}.
	\end{align*}
	Set 
	\begin{align*}
		(\forall k \in \mathbf{N}) \quad \alpha_{k} :=
		\begin{cases}
			c_{\alpha} \tau_{0} \quad &\text{if } k =0,\\
			c_{\alpha} \tau_{k-1} \quad &\text{if } k \geq 1,
		\end{cases}
	\end{align*}
	and 
	\begin{align*}
		\delta := \min \left\{1- \frac{L_{yx}}{c_{\alpha}}, 1- \left( c_{\alpha} L_{yx} \tau_{0} + 
		2L_{yy} \right)\sigma_{0} \right\}.
	\end{align*}

	We have the following assertions. 
	\begin{enumerate}
		\item 	\label{prop:finequalities:fxy} 
		Let $(x,y)$ be in $\mathcal{H}_{1} \times \mathcal{H}_{2}$.  
		Then for every $k \in \mathbf{N}$,
		\begin{align*}
			\sum^{k-1}_{i=0} t_{i} \left( f(x^{i+1}, y) -f(x,y^{i+1}) \right)
			\leq \frac{t_{0}}{2\tau_{0}} \norm{x-x^{0}}^{2} +\frac{t_{0}}{2\sigma_{0}} 
			\norm{y-y^{0}}^{2}.
		\end{align*}
		\item  	\label{prop:finequalities:fhatxy} Let $k \in \mathbf{N} \smallsetminus 
		\{0\}$.  
		Then
		\begin{align*}
			 -\frac{t_{0}}{2\tau_{0}} \norm{\hat{x}_{k}-x^{0}}^{2} 
			-\frac{t_{0}}{2\sigma_{0}} 	\norm{y^{*}-y^{0}}^{2} 
		&\leq 	\sum^{k-1}_{i=0} t_{i}  \left(f(\hat{x}_{k},\hat{y}_{k}) -f(x^{*},y^{*}) \right)\\ 
		& \leq 
			\frac{t_{0}}{2\tau_{0}} \norm{x^{*}-x^{0}}^{2} +\frac{t_{0}}{2\sigma_{0}} 
			\norm{\hat{y}_{k}-y^{0}}^{2}.
		\end{align*}
	\end{enumerate}
\end{proposition}

\begin{proof}
	\cref{prop:finequalities:fxy}:
	Let $k \in \mathbf{N}$. 
	In view of \cref{fact:preliminaries}\cref{fact:preliminaries:sigma}, we have that 
	\begin{align}\label{eq:prop:finequalities:fxy:sigma}
 \frac{1}{\sigma_{k}} - \theta_{k} \left(L_{yx} \alpha_{k} +L_{yy}\right)  \geq 
\frac{\sigma}{\sigma_{k}} >0.
	\end{align} 
	Applying \cref{fact:preliminaries}\cref{fact:preliminaries:f}  
	in the first inequality and employing 
	 \cref{fact:preliminaries}\cref{fact:preliminaries:bc} in the second inequality, 
	 we derive that
	 \begin{subequations} \label{eq:prop:finequalities:fxy}
	 	\begin{align}
	 		\sum^{k-1}_{i=0} t_{i} \left( f(x^{i+1}, y) -f(x,y^{i+1}) \right) 
	 		&\leq  \sum^{k-1}_{i=0} t_{i} \left(a_{i}(x,y) -b_{i+1}(x,y) -c_{i} \right)\\
	 		& \leq   \sum^{k-1}_{i=0} t_{i}a_{i}(x,y) -t_{i+1}a_{i+1}(x,y)\\
	 		& =   t_{0}a_{0}(x,y)  -t_{k}a_{k}(x,y). 
	 	\end{align}
	 \end{subequations}
	Furthermore, recalling the definitions of $a_{0}(x,y)$ and $a_{k}(x,y)$
	and the fact $(x^{-1},y^{-1}) =(x^{0},y^{0})$ and 
	$q_{0} :=\nabla_{y}\Phi(x^{0},y^{0}) -\nabla_{y}\Phi(x^{-1},y^{-1})=0$ 
	in the first equality, and applying \cref{fact:preliminaries}\cref{fact:preliminaries:qk} 
	in 	the first 	inequality below, we derive that
	\begin{align*}
		 &t_{0}a_{0}(x,y)  -t_{k}a_{k}(x,y)\\
=&\frac{t_{0}}{2\tau_{0}} \norm{x -x^{0}}^{2} +\frac{t_{0}}{2\sigma_{0}}
		 \norm{y -y^{0}}^{2} 
	-\frac{t_{k}}{2\tau_{k}} \norm{x -x^{k}}^{2} 
	-\frac{t_{k}}{2\sigma_{k}} \norm{y -y^{k}}^{2}\\ 
	&- t_{k}\theta_{k} \innp{q_{k}, y^{k}-y} 
	-t_{k}\theta_{k} \frac{L_{yx}}{2 \alpha_{k}} \norm{x^{k} -x^{k-1}}^{2} 
	-t_{k}\theta_{k} \frac{L_{yy}}{2} \norm{y^{k} -y^{k-1}}^{2} \\	
\leq &
\frac{t_{0}}{2\tau_{0}} \norm{x -x^{0}}^{2} +\frac{t_{0}}{2\sigma_{0}}
\norm{y -y^{0}}^{2} 
-\frac{t_{k}}{2\tau_{k}} \norm{x -x^{k}}^{2} 
-\frac{t_{k}}{2\sigma_{k}} \norm{y -y^{k}}^{2}\\ 
&+ t_{k}\theta_{k} \frac{L_{yx}}{2}\left(\alpha_{k} \norm{y-y_{k}}^{2} 
+ \frac{1}{\alpha_{k}}\norm{x^{k}-x^{k-1}}^{2}\right) 
+ t_{k}\theta_{k} \frac{L_{yy}}{2} \left(\norm{y-y^{k}}^{2} +\norm{y^{k} 
-y^{k-1}}^{2}\right)\\
&	-t_{k}\theta_{k} \frac{L_{yx}}{2 \alpha_{k}} \norm{x^{k} -x^{k-1}}^{2} 
-t_{k}\theta_{k} \frac{L_{yy}}{2} \norm{y^{k} -y^{k-1}}^{2} \\
=&	\frac{t_{0}}{2\tau_{0}} \norm{x -x^{0}}^{2} +\frac{t_{0}}{2\sigma_{0}}
\norm{y -y^{0}}^{2} 
-\frac{t_{k}}{2\tau_{k}} \norm{x -x^{k}}^{2} -\frac{t_{k}}{2} \left( \frac{1}{\sigma_{k}} - 
\theta_{k} \left(L_{yx} \alpha_{k} 
+L_{yy}\right) \right)\norm{y -y^{k}}^{2}\\
\leq &	\frac{t_{0}}{2\tau_{0}} \norm{x -x^{0}}^{2} +\frac{t_{0}}{2\sigma_{0}}
\norm{y -y^{0}}^{2}, 
	\end{align*}
where in the last inequality, we use \cref{eq:prop:finequalities:fxy:sigma}.
Altogether, we have the desired result
\begin{align*}
	\sum^{k-1}_{i=0} t_{i} \left( f(x^{i+1}, y) -f(x,y^{i+1}) \right) 
\leq \frac{t_{0}}{2\tau_{0}} \norm{x -x^{0}}^{2} +\frac{t_{0}}{2\sigma_{0}}
\norm{y -y^{0}}^{2}.	
\end{align*}	
	
	\cref{prop:finequalities:fhatxy}: 
Combine \cref{eq:lemma:fxkykx*y*:x} in \cref{fact:fxkykx*y*} and the result 
	obtained from  \cref{prop:finequalities:fxy} above to observe that
	\begin{align*}
	\sum^{k-1}_{i=0} t_{i} \left( 	 f(\hat{x}_{k},\hat{y}_{k}) -f(x^{*},y^{*}) \right) 
		 \leq &  
		 \sum^{k-1}_{j=0} t_{j} \left( f(x^{j+1}, \hat{y}_{k})  
		 -f(x^{*},y^{j+1}) \right)\\
		 \leq &
		 \frac{t_{0}}{2\tau_{0}} \norm{x^{*}-x^{0}}^{2} +\frac{t_{0}}{2\sigma_{0}} 
		 \norm{\hat{y}_{k}-y^{0}}^{2}.
	\end{align*}
Similarly, applying \cref{eq:lemma:fxkykx*y*:y} in \cref{fact:fxkykx*y*} and the result 
obtained from  \cref{prop:finequalities:fxy} above, we have that 
\begin{align*}
	\sum^{k-1}_{i=0} t_{i} \left( 	 f(\hat{x}_{k},\hat{y}_{k}) -f(x^{*},y^{*}) \right) 
	 \geq &
	 -  \sum^{k-1}_{j=0} t_{j} \left( f(x^{j+1},y^{*})  
	 -f(\hat{x}_{k},y^{j+1}) \right)\\
	 \geq &-\frac{t_{0}}{2\tau_{0}} \norm{\hat{x}_{k}-x^{0}}^{2} 
	 -\frac{t_{0}}{2\sigma_{0}} 	\norm{y^{*}-y^{0}}^{2}.
\end{align*}
Altogether, we obtain the required result. 
\end{proof}

Below, we show the convergence of the OGAProx without strongly convexity or 
concavity assumption. 
In fact, the assumption of \cref{theorem:convexconcave} is the same as that of
\cite[Theorem~9]{BotCsetnekSedlmayer2022accelerated}
which proves the convergence of the sequence of iterations generated by the OGAProx 
and the convergence of the min-max gap evaluated at the associated ergodic 
sequences
under the convex-concave setting. 
\begin{theorem}
	\label{theorem:convexconcave}
		Let  $\nu = 0$, let $c_{\alpha} > L_{yx} \geq 0$, and  let $\tau$ and $\sigma $ 
		be 
		in 
	$\mathbf{R}_{++}$ such that 
	\begin{align*}
		\left(c_{\alpha} L_{yx} \tau +2 L_{yy} \right)\sigma  <1.
	\end{align*}
Let $(\forall k \in \mathbf{N})$ $\tau_{k}\equiv \tau  $, 
$\sigma_{k}\equiv \sigma $, 
and $\theta_{k}\equiv 1$. 
	Then for every $k \in \mathbf{N} \smallsetminus \{0\}$,
\begin{align*}
	- \frac{1}{k} \left( \frac{t_{0}}{2\tau_{0}} \norm{\hat{x}_{k}-x^{0}}^{2} 
	+\frac{t_{0}}{2\sigma_{0}} 	\norm{y^{*}-y^{0}}^{2} \right)
	&\leq 	f(\hat{x}_{k},\hat{y}_{k}) -f(x^{*},y^{*}) \\ 
	& \leq \frac{1}{k}
	\left( \frac{t_{0}}{2\tau_{0}} \norm{x^{*}-x^{0}}^{2} +\frac{t_{0}}{2\sigma_{0}} 
	\norm{\hat{y}_{k}-y^{0}}^{2} \right).
\end{align*}
	Consequently, $\left( f(\hat{x}^{k+1}, \hat{y}^{k+1}) \right)_{k \in \mathbf{N}}$ 
	converges to 
	$f(x^{*},y^{*})$ with a convergence rate of
	order $\mathcal{O}\left(\frac{1}{k}\right)$.
\end{theorem}

\begin{proof}
	In view of \cite[Proposition~8]{BotCsetnekSedlmayer2022accelerated},
	the assumptions above on the parameters $\left(\sigma_{k}\right)_{k \in \mathbf{N}}$,
	$\left(\tau_{k}\right)_{k \in \mathbf{N}}$, and $\left(\theta_{k}\right)_{k \in 
		\mathbf{N}}$ are exactly the requirements of the parameters 
	in \cref{prop:finequalities} when $\nu =0$.
	
	According to \cite[Theorem~9]{BotCsetnekSedlmayer2022accelerated},
	$\left( (x^{k},y^{k}) \right)_{k \in \mathbf{N}}$ weakly converges to $(x^{*},y^{*})$,
	which implies that $\left( (x^{k},y^{k}) \right)_{k \in \mathbf{N}}$  is bounded. 
	Due to \cref{eq:hatxy}, 
	the boundedness of $\left( (x^{k},y^{k}) \right)_{k \in \mathbf{N}}$ 
	guarantees the boundedness   of $\left( (\hat{x}^{k+1}, \hat{y}^{k+1}) \right)_{k \in 
	\mathbf{N}}$
	
	Because $(\forall k \in \mathbf{N})$  $\theta_{k}\equiv 1$, 
	we know that $(\forall k \in \mathbf{N})$ $t_{k} = 
	\frac{\theta_{0}}{\theta_{0}\theta_{1}\cdots 
		\theta_{k}} \equiv 1$. Hence, we have that 
	\begin{align*}
		(\forall k \in \mathbf{N} \smallsetminus \{0\}) \quad 
\sum^{k-1}_{j=0} t_{j} =k.
	\end{align*}
Combine this result with \cref{prop:finequalities}\cref{prop:finequalities:fhatxy}
to deduce that
	\begin{align*}
	-\frac{t_{0}}{2\tau_{0}} \norm{\hat{x}_{k}-x^{0}}^{2} 
	-\frac{t_{0}}{2\sigma_{0}} 	\norm{y^{*}-y^{0}}^{2} 
	&\leq 	k \left(f(\hat{x}_{k},\hat{y}_{k}) -f(x^{*},y^{*}) \right)\\ 
	& \leq 
	\frac{t_{0}}{2\tau_{0}} \norm{x^{*}-x^{0}}^{2} +\frac{t_{0}}{2\sigma_{0}} 
	\norm{\hat{y}_{k}-y^{0}}^{2},
\end{align*}
which, combining with the boundedness of $\left( (\hat{x}^{k+1}, \hat{y}^{k+1}) 
\right)_{k \in \mathbf{N}}$, ensures the required results. 
\end{proof}

In the following result, we present the convergence of the OGAProx with 
the associated function being convex-strongly concave. 
Note that the assumption of \cref{theorem:convexstronglyconcave}
is exactly the same as \cite[Theorem~12]{BotCsetnekSedlmayer2022accelerated}
which shows the convergence of the sequence of iterations generated by 
the OGAProx 
and the convergence of the min-max gap evaluated at the associated ergodic 
sequences
under the convex-strongly concave setting. 
\begin{theorem}
	\label{theorem:convexstronglyconcave}
Let $\nu >0$,	let $c_{\alpha} > L_{yx} \geq 0$, let $\theta_{0}=1$,
	and  let $\tau_{0}$ and $\sigma_{0} $ be 
	in 
	$\mathbf{R}_{++}$ such that 
	\begin{align*}
		\left(c_{\alpha} L_{yx} \tau_{0} +2 L_{yy} \right)\sigma_{0}  <1 
		\quad \text{and} \quad
		0 < \sigma_{0} \leq \frac{9 +3\sqrt{13}}{2 \nu}.
	\end{align*}
Define
\begin{align*}
	(\forall k \in \mathbf{N}) \quad \theta_{k+1} := \frac{1}{\sqrt{1+ \nu \sigma_{k}}}, \quad
	\tau_{k+1} := \frac{\tau_{k}}{\theta_{k+1}}, \quad
	\text{and} \quad \sigma_{k+1} := \theta_{k+1}\sigma_{k}.
\end{align*} 
	The following results hold. 
	\begin{enumerate}
		\item \label{theorem:convexstronglyconcave:bounded}
		$\left( (x^{k},y^{k}) \right)_{k \in \mathbf{N}}$ 
		and $\left( (\hat{x}^{k+1}, \hat{y}^{k+1}) \right)_{k \in \mathbf{N}}$  are bounded. 
		\item \label{theorem:convexstronglyconcave:convergence}
		For every $k \in \mathbf{N} \smallsetminus \{0\}$,
		\begin{align*}
			- \frac{6}{\nu \sigma_{0} k^{2}} \left(\frac{1}{\tau_{0}} 
		\norm{\hat{x}_{k}-x^{0}}^{2} 
		-\frac{1}{\sigma_{0}} 	\norm{y^{*}-y^{0}}^{2}  \right)
		&\leq 	 f(\hat{x}_{k},\hat{y}_{k}) -f(x^{*},y^{*})  \\ 
		& \leq 
		\frac{6}{\nu \sigma_{0} k^{2}}
		\left(	\frac{1}{\tau_{0}} \norm{x^{*}-x^{0}}^{2} +\frac{1}{\sigma_{0}} 
		\norm{\hat{y}_{k}-y^{0}}^{2} \right).
		\end{align*}
		Consequently, $\left( f(\hat{x}^{k+1}, \hat{y}^{k+1}) \right)_{k \in \mathbf{N}}$ 
		converges to 
		$f(x^{*},y^{*})$ with a convergence rate order 
	$\mathcal{O}\left(\frac{1}{k^{2}}\right)$.
	\end{enumerate}
\end{theorem}

\begin{proof}
\cref{theorem:convexstronglyconcave:bounded}:	Let $k \in \mathbf{N} \smallsetminus 
\{0\}$.
Because $(x^{*},y^{*})$ is a saddle-point of $f$,	
in view of \cref{fact:preliminaries}\cref{fact:preliminaries:fhatxy}, we have that 
$0 \leq \sum^{k-1}_{i=0} t_{i} \left(  f(\hat{x}_{k}, y^{*}) -f(x^{*},\hat{y}_{k})  \right) \leq  
\frac{t_{0}}{2\tau_{0}} \norm{x^{*} -x^{0}}^{2} +\frac{t_{0}}{2\sigma_{0}}
\norm{y^{*} -y^{0}}^{2} 
-\frac{t_{k}}{2\tau_{k}} \norm{x^{*} -x^{k}}^{2} 
-\frac{t_{k}}{2} \left( \frac{1}{\sigma_{k}} - 
\theta_{k} \left(L_{yx} \alpha_{k} 
+L_{yy}\right) \right)\norm{y^{*} -y^{k}}^{2}$,
which, combining with the fact  $t_{0}=1$,  implies that 
	 \begin{align}\label{eq:theorem:convexstronglyconcave}
	   \frac{1}{2\tau_{0}} \norm{x^{*} -x^{0}}^{2} +\frac{1}{2\sigma_{0}}
	 	 \norm{y^{*} -y^{0}}^{2}  \geq 
	 	 \frac{t_{k}}{2\tau_{k}} \norm{x^{*} -x^{k}}^{2} 
	 	 +\frac{t_{k}}{2} \left( \frac{1}{\sigma_{k}} -  \theta_{k} \left(L_{yx} \alpha_{k} 
	 	 +L_{yy}\right) \right)\norm{y^{*} -y^{k}}^{2}.
	 \end{align}
 In view of 
 \cref{fact:preliminaries}\cref{fact:preliminaries:sigma}$\&$\cref{fact:preliminaries:tk}, 
 we know that $ \frac{1}{\sigma_{k}} -  \theta_{k} \left(L_{yx} \alpha_{k} 
 +L_{yy}\right) \geq \frac{\delta}{\sigma_{k}}$ and 
 $\frac{t_{k}}{\tau_{k}} =\frac{1}{\tau_{0}}$. 
 Moreover, via \cite[Proposition~11]{BotCsetnekSedlmayer2022accelerated},
 we know that $\frac{t_{k}}{\sigma_{k}} = \frac{t_{k}}{\tau_{k}} \frac{\tau_{k}}{\sigma_{k} 
 } \geq \frac{1}{\tau_{0}} \frac{\nu^{2} \tau_{0} \sigma_{0}}{9} k^{2} $. 
 Combine these results with \cref{eq:theorem:convexstronglyconcave}
 to derive that
 \begin{align*}
 	 &\frac{1}{2\tau_{0}} \norm{x^{*} -x^{0}}^{2} +\frac{1}{2\sigma_{0}}
 	\norm{y^{*} -y^{0}}^{2}  \\
 	\geq & 
 	\frac{t_{k}}{2\tau_{k}} \norm{x^{*} -x^{k}}^{2} 
 	+\frac{t_{k}}{2} \left( \frac{1}{\sigma_{k}} -  \theta_{k} \left(L_{yx} \alpha_{k} 
 	+L_{yy}\right) \right)\norm{y^{*} -y^{k}}^{2}\\
 	\geq &
 	\frac{1}{2\tau_{0}} \norm{x^{*} -x^{k}}^{2} 
 	+ \frac{t_{k}}{2} \frac{\delta}{\sigma_{k}} \norm{y^{*} -y^{k}}^{2}\\
 	\geq & \frac{1}{2\tau_{0}} \norm{x^{*} -x^{k}}^{2} 
 	+  \frac{1}{\tau_{0}} \frac{\nu^{2} \tau_{0} \sigma_{0}}{9} k^{2} \norm{y^{*} 
 	-y^{k}}^{2},
 \end{align*}
which, via \cref{eq:hatxy},  implies that $\left( (x^{k},y^{k}) \right)_{k \in \mathbf{N}}$ 
and $\left( (\hat{x}^{k+1}, \hat{y}^{k+1}) \right)_{k \in \mathbf{N}}$  are bounded. 

\cref{theorem:convexstronglyconcave:convergence}: 
Applying  \cref{fact:preliminaries}\cref{fact:preliminaries:tk} again and 
\cite[Inequality~(39)]{BotCsetnekSedlmayer2022accelerated}, we know that
\begin{align*}
	\sum^{k-1}_{i=0} t_{k} =\frac{1}{\tau_{0}} \sum^{k-1}_{i=0} \tau_{k} 
	\geq \frac{1}{\tau_{0}}  \frac{\nu \tau_{0}\sigma_{0}}{3}\sum^{k-1}_{i=0} k= 
	\frac{\nu \sigma_{0}}{6}  k(k-1).
\end{align*} 
Clearly, if $k \in \mathbf{N} \smallsetminus \{0,1\}$, then $k \geq 2$,  
$k-1 \geq \frac{k}{2}$, 
and $\frac{\nu \sigma_{0}}{6}  k(k-1) \geq \frac{\nu \sigma_{0}}{12}  k^{2}$. 
Hence, we have that 
\begin{align*}
\left(\forall k \in \mathbf{N} \smallsetminus \{0,1\} \right) \quad 	
\frac{1}{\sum^{k-1}_{i=0} t_{k}} \leq \frac{12}{\nu \sigma_{0}  k^{2}}.
\end{align*} 
Combine results above  
with \cref{prop:finequalities}\cref{prop:finequalities:fhatxy}
to obtain that
\begin{align*}
	  f(\hat{x}_{k},\hat{y}_{k}) -f(x^{*},y^{*})  
	& \leq \frac{1}{\sum^{k-1}_{i=0} t_{i}}
	\left( \frac{t_{0}}{2\tau_{0}} \norm{x^{*}-x^{0}}^{2} +\frac{t_{0}}{2\sigma_{0}} 
	\norm{\hat{y}_{k}-y^{0}}^{2} \right)\\
	&\leq \frac{12}{\nu \sigma_{0}  k^{2}}
		\left( \frac{t_{0}}{2\tau_{0}} \norm{x^{*}-x^{0}}^{2} +\frac{t_{0}}{2\sigma_{0}} 
	\norm{\hat{y}_{k}-y^{0}}^{2} \right)
\end{align*}
and 
\begin{align*}
	 f(\hat{x}_{k},\hat{y}_{k}) -f(x^{*},y^{*})
	 &\geq -\frac{1}{\sum^{k-1}_{i=0} t_{i}} 
	 \left(  \frac{t_{0}}{2\tau_{0}} \norm{\hat{x}_{k}-x^{0}}^{2} 
	 +\frac{t_{0}}{2\sigma_{0}} 	\norm{y^{*}-y^{0}}^{2} \right)\\
	 &\geq -\frac{12}{\nu \sigma_{0}  k^{2}}
	 \left(  \frac{t_{0}}{2\tau_{0}} \norm{\hat{x}_{k}-x^{0}}^{2} 
	 +\frac{t_{0}}{2\sigma_{0}} 	\norm{y^{*}-y^{0}}^{2} \right).
\end{align*}
Recall that $t_{0}=1$. Therefore, we obtain that 
\begin{align*}
	- \frac{6 }{\nu \sigma_{0} k^{2}} \left(\frac{1}{\tau_{0}} 
	\norm{\hat{x}_{k}-x^{0}}^{2} 
	-\frac{1}{\sigma_{0}} 	\norm{y^{*}-y^{0}}^{2}  \right)
	&\leq 	 f(\hat{x}_{k},\hat{y}_{k}) -f(x^{*},y^{*})  \\ 
	& \leq 
	\frac{6 }{\nu \sigma_{0} k^{2}}
\left(	\frac{1}{\tau_{0}} \norm{x^{*}-x^{0}}^{2} +\frac{1}{\sigma_{0}} 
	\norm{\hat{y}_{k}-y^{0}}^{2} \right),
\end{align*}
which, combining with the boundedness of $\left( (\hat{x}^{k+1}, \hat{y}^{k+1}) 
\right)_{k \in \mathbf{N}}$, ensures the required results. 
\end{proof}

\subsection{Strongly Convex-Strongly Concave Setting}

In this subsection, we assume additionally that 
$(\forall y \in \dom g)$ $\Phi(\cdot, y) : \mathcal{H}_{1} \to \mathbf{R} \cup \{+\infty\}$
is $\mu$-strongly convex with $\mu >0$ and that the function $g$ is convex with 
modulus $\nu >0$.
That means we assume that the function 
$\left(\forall (x,y) \in \mathcal{H}_{1} \times \mathcal{H}_{2}\right)$ $f(x,y)=\Phi(x,y) 
-g(y)$  is strongly convex-strongly concave in this subsection.

\begin{lemma}
	\label{lemma:sconvexsconcave}
	Let $(\forall k \in \mathbf{N})$ $\sigma_{k} \equiv \sigma \in \mathbf{R}_{++}$, 
	$\tau_{k} \equiv \tau \in \mathbf{R}_{++}$, and 
	$\theta_{k} \equiv \theta \in (0,1)$ such that 
	\begin{align*}
		 1 +\mu \tau = \frac{1}{\theta} \quad \text{and} \quad 1+\nu \sigma =\frac{1}{\theta}.
	\end{align*}
Suppose that there exists $\alpha \in \mathbf{R}_{++}$ such that
\begin{align}\label{eq:lemma:sconvexsconcave}
	\frac{L_{yx}}{\alpha} \leq \frac{1}{\tau}, 
	\quad L_{yy} \leq \frac{1 -\theta \sigma (\alpha L_{yx + L_{yy}})}{\sigma}, 
	\quad \text{and} \quad  
	1-\theta \sigma (\alpha L_{yx} +L_{yy}) >0.
\end{align}
Set $\tilde{\sigma}:= \frac{\sigma}{1 -\theta \sigma (\alpha L_{yx} +L_{yy})}$. The 
following statements hold. 
\begin{enumerate}
	\item \label{lemma:sconvexsconcave:tk}
	$(\forall k \in \mathbf{N})$ $t_{k} =\frac{1}{\theta^{k}}$.
	
	\item \label{lemma:sconvexsconcave:sum} 
	 Let $(x,y)$ be in $\mathcal{H}_{1} \times 	\mathcal{H}_{2}$ and  
	 let $k \in \mathbf{N} \smallsetminus \{0\}$. Then we have that
	\begin{align*}
		&\sum^{k-1}_{i=0} t_{i} \left(f(x^{k+1},y) -f(x,y^{k+1}) \right) \\
		\leq & \frac{1}{2\tau} \norm{x -x^{0}}^{2} + \frac{1}{2\sigma} \norm{y -y^{0}}^{2} 
		-\frac{1}{\theta^{k}} \frac{1}{2 \tau} \norm{x -x^{k}}^{2} -\frac{1}{\theta^{k}}\frac{1 
		-\theta \sigma(\alpha L_{yx}+L_{yy})}{2 \sigma} \norm{y -y^{k}}^{2}\\
	&-\frac{1}{2\theta^{k-1}} \left(\frac{1}{\tau} -\frac{L_{yx}}{\alpha}\right)\norm{x^{k} 
	-x^{k-1}}^{2}
	-\frac{1}{2\theta^{k-1}} \left( \frac{1}{\tilde{\sigma}} -L_{yy}\right) \norm{y^{k} 
	-y^{k-1}}^{2}.
	\end{align*}  
	\item \label{lemma:sconvexsconcave:leq} 
	Let $(x,y)$ be in $\mathcal{H}_{1} \times 	\mathcal{H}_{2}$ and  
	let $k \in \mathbf{N} \smallsetminus 	\{0\}$. Then we have that
	\begin{align*}
		&\sum^{k-1}_{i=0} t_{i} \left(f(x^{k+1},y) -f(x,y^{k+1}) \right) \\
		\leq &\sum^{k-1}_{i=0} t_{i} \left(f(x^{k+1},y) -f(x,y^{k+1}) \right) 
		+\frac{1}{\theta^{k}} \frac{1}{2 \tau} \norm{x -x^{k}}^{2} 
		+\frac{1}{\theta^{k}}\frac{1}{2 \tilde{\sigma}} \norm{y -y^{k}}^{2}\\
		\leq & \frac{1}{2\tau} \norm{x -x^{0}}^{2} + \frac{1}{2\sigma} \norm{y -y^{0}}^{2}. 
	\end{align*}  
\end{enumerate}
\end{lemma}

\begin{proof}
	\cref{lemma:sconvexsconcave:tk}: Because $(\forall k \in \mathbf{N})$  $\theta_{k} 
	\equiv \theta \in (0,1)$, we have that 
	\begin{align*}
		(\forall k \in \mathbf{N}) \quad t_{k} = \frac{\theta_{0}}{\theta_{0}\theta_{1}\cdots 
			\theta_{k}} = \frac{\theta}{\theta^{k+1}}=\frac{1}{\theta^{k}}.
	\end{align*}
	
\cref{lemma:sconvexsconcave:sum}: 
This is a direct result of \cref{lemma:sconvexsconcave:tk} and
\cite[Inequality~(46)]{BotCsetnekSedlmayer2022accelerated}.

\cref{lemma:sconvexsconcave:leq}:  Based on \cref{eq:lemma:sconvexsconcave}, we 
know that 
\begin{align*}
	\frac{1}{\tau} -\frac{L_{yx}}{\alpha} \geq 0 \quad \text{and} \quad 
	\frac{1}{\tilde{\sigma}} -L_{yy} 
	= \frac{1 -\theta \sigma (\alpha L_{yx + L_{yy}})}{\sigma}- L_{yy} \geq 0.
\end{align*}
Combine this result with
\cref{lemma:sconvexsconcave:sum} to derive that
	\begin{align*}
	&\sum^{k-1}_{i=0} t_{i} \left(f(x^{k+1},y) -f(x,y^{k+1}) \right) \\
	\leq & \frac{1}{2\tau} \norm{x -x^{0}}^{2} + \frac{1}{2\sigma} \norm{y -y^{0}}^{2} 
	-\frac{1}{\theta^{k}} \frac{1}{2 \tau} \norm{x -x^{k}}^{2} 
	-\frac{1}{\theta^{k}}\frac{1}{2 \tilde{\sigma}} \norm{y -y^{k}}^{2},
\end{align*} 
which ensures the required result clearly. 
\end{proof}

In \cref{theorem:stronglyconvexstronglyconcave} below, we show the convergence 
of OGAProx with associated function being strongly convex-strongly concave.
Notice that the assumption of  \cref{theorem:stronglyconvexstronglyconcave}  below
is the same as that of \cite[Theorem~14]{BotCsetnekSedlmayer2022accelerated}
which shows the convergence of the sequence of iterations generated by the OGAProx
and  the convergence of the min-max gap evaluated at the associated ergodic 
sequences
under the strongly convex-strongly concave setting. 
 \begin{theorem} \label{theorem:stronglyconvexstronglyconcave}
 	Let $\alpha \in \mathbf{R}_{++}$. 
 	Set  $\tilde{\theta} :=\max \{ \frac{L_{yx}}{\alpha \mu +L_{yx}}, 
 	\frac{\alpha L_{yx}+2L_{yy}}{\nu +\alpha L_{yx} +2L_{yy}}  \}$. Let $\theta \in 
 	(\tilde{\theta},1) \subseteq [0,1)$. Let 
 	\begin{align*}
 		(\forall k \in \mathbf{N}) \quad \sigma_{k}\equiv \sigma =\frac{1}{\nu} 
 		\frac{1-\theta}{\theta}, \quad \tau_{k} \equiv \tau 
 		=\frac{1}{\mu}\frac{1-\theta}{\theta},  \quad  \text{and} \quad \theta_{k} \equiv 
 		\theta. 
 	\end{align*}
 	Then 
 	\begin{align*}
   -\theta^{k-1}   \left( \frac{1}{2\tau} \norm{\hat{x}_{k} 
 	-x^{0}}^{2} - \frac{1}{2\sigma} \norm{y^{*} 
 	-y^{0}}^{2} \right) &\leq		f(\hat{x}_{k},\hat{y}_{k}) -f(x^{*},y^{*}) \\
 &  \leq  \theta^{k-1} 
 	\left( \frac{1}{2\tau} \norm{x^{*} -x^{0}}^{2} + 
 			\frac{1}{2\sigma} \norm{\hat{y}_{k} 
 				-y^{0}}^{2} \right).
 	\end{align*}
 Consequently, 
the sequence $\left( f(\hat{x}^{k+1}, \hat{y}^{k+1}) \right)_{k \in \mathbf{N}}$ 
linearly  converges to 
 $f(x^{*},y^{*})$ with a convergence rate order 
 $\mathcal{O}\left(\theta^{k}\right)$.
 \end{theorem}

\begin{proof}
	Due to  \cite[Proposition~13]{BotCsetnekSedlmayer2022accelerated}, the assumption
	above on the parameters $\left(\sigma_{k}\right)_{k \in \mathbf{N}}$,
	$\left(\tau_{k}\right)_{k \in \mathbf{N}}$, and $\left(\theta_{k}\right)_{k \in 
	\mathbf{N}}$ satisfy related requirements in \cref{lemma:sconvexsconcave}.

	 In view of \cite[Theorem~14]{BotCsetnekSedlmayer2022accelerated},
	 we know that  
	 $\left( (x^{k},y^{k}) \right)_{k \in \mathbf{N}}$ linearly converges to $(x^{*},y^{*})$,
	 which, via \cref{eq:hatxy}, guarantees the boundedness of  $\left( (x^{k},y^{k}) 
	 \right)_{k \in \mathbf{N}}$ and $\left( (\hat{x}^{k+1}, \hat{y}^{k+1}) 
	 \right)_{k \in \mathbf{N}}$.
	 
	 Let $k$ be in $\mathbf{N} \smallsetminus \{0\}$.
	Combine \cref{fact:fxkykx*y*} and 
	\cref{lemma:sconvexsconcave}\cref{lemma:sconvexsconcave:leq}
	to derive that
	\begin{subequations}\label{theorem:stronglyconvexstronglyconcave:leq}
	\begin{align}
  f(\hat{x}_{k},\hat{y}_{k}) -f(x^{*},y^{*})  
	&\leq \frac{1}{\sum^{k-1}_{i=0} t_{i} } \sum^{k-1}_{j=0} t_{j} \left( f(x^{j+1}, 
	\hat{y}_{k})  
	-f(x^{*},y^{j+1}) \right)\\
	&\leq \frac{1}{\sum^{k-1}_{i=0} t_{i} } \left( \frac{1}{2\tau} \norm{x^{*} -x^{0}}^{2} + 
	\frac{1}{2\sigma} \norm{\hat{y}_{k} 
		-y^{0}}^{2} \right)
\end{align}		
	\end{subequations}
and 
	\begin{subequations}\label{theorem:stronglyconvexstronglyconcave:geq}
\begin{align}
 f(\hat{x}_{k},\hat{y}_{k}) -f(x^{*},y^{*})  
	&\geq -\frac{1}{\sum^{k-1}_{i=0} t_{i} } \sum^{k-1}_{j=0} t_{j} \left( f(x^{j+1},y^{*})  
	-f(\hat{x}_{k},y^{j+1}) \right)\\
	&\geq -\frac{1}{\sum^{k-1}_{i=0} t_{i} }
	 \left( \frac{1}{2\tau} \norm{\hat{x}_{k} -x^{0}}^{2} - \frac{1}{2\sigma} \norm{y^{*} 
		-y^{0}}^{2} \right).
\end{align}	
\end{subequations}

Because $\theta \in (0,1)$, we know that  $0<\theta^{k} \leq \theta<1$ and 
$\frac{1-\theta^{k} }{1-\theta}  \geq 1$. This result together with
  \cref{lemma:sconvexsconcave}\cref{lemma:sconvexsconcave:tk} yields that
\begin{align}\label{theorem:stronglyconvexstronglyconcave:sumtk}
	\sum^{k-1}_{i=0} t_{i}  = \sum^{k-1}_{i=0} \frac{1}{\theta^{i}} =\frac{1}{\theta^{k-1}}
	 \sum^{k-1}_{i=0} \theta^{i} =\frac{1}{\theta^{k-1}} \frac{1-\theta^{k} }{1-\theta} \geq 
	 \frac{1}{\theta^{k-1}}. 
\end{align}
Combine \cref{theorem:stronglyconvexstronglyconcave:leq}, 
\cref{theorem:stronglyconvexstronglyconcave:geq}, and 
\cref{theorem:stronglyconvexstronglyconcave:sumtk} to obtain that 
\begin{align*} 
 	f(\hat{x}_{k},\hat{y}_{k}) -f(x^{*},y^{*})  
 	&\leq \frac{1}{\sum^{k-1}_{i=0} t_{i} } \left( \frac{1}{2\tau} \norm{x^{*} -x^{0}}^{2} + 
 	\frac{1}{2\sigma} \norm{\hat{y}_{k} 
 		-y^{0}}^{2} \right)\\
&\leq \theta^{k-1} \left( \frac{1}{2\tau} \norm{x^{*} -x^{0}}^{2} + 
\frac{1}{2\sigma} \norm{\hat{y}_{k} 
	-y^{0}}^{2} \right)
\end{align*}
and 
\begin{align*}
	f(\hat{x}_{k},\hat{y}_{k}) -f(x^{*},y^{*})  
	&\geq -\frac{1}{\sum^{k-1}_{i=0} t_{i} }
	\left( \frac{1}{2\tau} \norm{\hat{x}_{k} -x^{0}}^{2} - \frac{1}{2\sigma} \norm{y^{*} 
		-y^{0}}^{2} \right)\\
	&\geq  -\theta^{k-1}   \left( \frac{1}{2\tau} \norm{\hat{x}_{k} 
		-x^{0}}^{2} - \frac{1}{2\sigma} \norm{y^{*} 
		-y^{0}}^{2} \right). 
\end{align*}
which, combining with the boundedness of $\left( (\hat{x}^{k+1}, \hat{y}^{k+1}) 
\right)_{k \in \mathbf{N}}$, ensures the required results. 
\end{proof}

 \section*{Acknowledgments}
Hui Ouyang thanks Professor Boyd Stephen for his insight and expertise comments 
 on the topic of saddle-point problems
and all unselfish support. 
Hui Ouyang acknowledges 
the Natural Sciences and Engineering Research Council of Canada (NSERC), 
[funding reference number PDF – 567644 – 2022].

 \addcontentsline{toc}{section}{References}
\bibliographystyle{abbrv}
\bibliography{fconvergence}

\end{document}